\documentclass[12pt]{amsart}
\usepackage{graphicx}
\usepackage{amsmath,amssymb,amsthm}
\usepackage[final]{showkeys}  
\usepackage[english]{babel}
\newcommand{\Cx}{\mathbb{C}}

\newcommand{\Rl}{\mathbb{R}}

\newcommand{\pard}[2]{\frac{\partial #1}{\partial #2}}
\newcommand{\pardd}[2]{\frac{\partial^2 #1}{\partial #2^2}}
\newcommand{\pards}[3]{\frac{\partial^2 #1}{\partial #2 \partial #3}}

\DeclareMathOperator{\grad}{grad}
\DeclareMathOperator*{\dv}{d}

\DeclareMathOperator{\spa}{span}
\DeclareMathOperator{\sign}{sg}
\DeclareMathOperator*{\tr}{Tr}

\newcommand{\scal}[1]{\langle #1\rangle}
\newcommand{\lie}[1]{\left[#1\right]}

\providecommand{\abs}[1]{\lvert#1\rvert}
\providecommand{\norm}[1]{\lVert#1\rVert}

\newtheorem{lemma}{Lemma}[section]
\newtheorem{theorem}{Theorem}[section]

\newtheorem{remark}{Remark}[section]

\begin{document}

\title[Special Lagrangian 4-folds with $SO(2)\rtimes S_3$-Symmetry]{Special Lagrangian 4-folds with $SO(2)\rtimes S_3$-Symmetry in Complex Space Forms}

\email[F. Dillen]{franki.dillen@wis.kuleuven.be}
\email[C. Scharlach]{schar@math.tu-berlin.de}
\email[K. Schoels]{Kristof.Schoels@wis.kuleuven.be}
\email[L. Vrancken]{luc.vrancken@univ-valenciennes.fr}

\author[F. Dillen]{Franki Dillen}
\address[F. Dillen, K. Schoels, L. Vrancken]
{Department of Mathematics\\ Katholieke Universiteit Leuven\\
Celestijnenlaan 200 B, Box 2400\\BE-3001 Leuven\\ Belgium}

\author[C. Scharlach]{Christine Scharlach}
\address[C. Scharlach]{TU Berlin\\Fakult\"at II\\
Institut f\"ur Mathematik\\
Strasse des 17. Juni 136\\
D-10623 Berlin\\Germany}

\author[K. Schoels]{Kristof Schoels}

\author[L. Vrancken]{Luc Vrancken}
\address[L. Vrancken]{Univ. Lille Nord de France, F-59000 Lille, France}
\address[L. Vrancken]{LAMAV\\ISTV2\\ Universit\'e de Valenciennes\\
Campus du Mont Houy\\ F-59313 Valenciennes Cedex 9\\ France}

\begin{abstract}
In this article we obtain a classification of special Lagrangian submanifolds in complex space forms subject to an $SO(2)\rtimes S_3$-symmetry on the second fundamental form. The algebraic structure of this form has been obtained by Marianty Ionel in \cite{mariant}. However, the classification of special Lagrangian submanifolds in $\mathbb{C}^4$ having this $SO(2)\rtimes S_3$ symmetry in \cite{mariant} is incomplete. In this paper we give a complete classification of such submanifolds, and extend the classification to special Lagrangian submanifolds of arbitrary complex space forms with $SO(2)\rtimes S_3$-symmetry.
\end{abstract}

\maketitle

\section{Introduction}
A space $(N,J,g)$ is called a Hermitian manifold with complex structure $J$ and Riemannian metric $g$, if $g(JX,JY)=g(X,Y)$ for all $X$ and $Y$. The $(0,2)$-tensor $\omega(X,Y)=g(X,JY)$ is its symplectic form. If $\omega$ is closed, then $(N,J,g)$ is said to be a K\"ahler manifold. In this case the Levi-Civita connection $D$ of $g$ satisfies $D\omega=0$ as well, see \cite{nomkob2}. A complex space form is a K\"ahler manifold for which the curvature tensor is given by
\begin{equation}
R(X,Y)Z=\epsilon \left(X\wedge Y+JX\wedge JY+2g(X,JY)J\right) Z, \label{csf}
\end{equation}
where $\epsilon$ is a real constant and $X\wedge Y$ is defined as
\begin{equation*}
(X\wedge Y)Z=g(Y,Z)X-g(X,Z)Y.
\end{equation*}
Every complete, simply connected complex space form of dimension $n$ with constant holomorphic sectional curvature $4\epsilon$ is isometric to one of the following manifolds:
\begin{enumerate}
\item the standard complex space $\Cx^n$ when $\epsilon=0$,
\item the complex projective space $\Cx P^n(4\epsilon)$ when $\epsilon>0$,
\item the complex hyperbolic space $\Cx H^n(4\epsilon)$ when $\epsilon<0$.
\end{enumerate}
Because we consider submanifolds of a complex space form locally, we can restrict ourselves to those ambient spaces. By rescaling, we can even assume that $\epsilon=0,1,-1$.

A Lagrangian submanifold $M$ of a K\"ahler manifold $(N,J,g)$ is a submanifold such that $\omega$ vanishes identically on $M$ and the (real) dimension of $M$ is half the (complex) dimension of $N$, see \cite{audin}. This implies that $J$ induces an orthogonal isomorphism between the tangent and the normal bundle on the submanifold. The Gauss formula is given by
\begin{equation*}
D_XY=\nabla_XY+h(X,Y)=\nabla_XY+JA(X,Y),
\end{equation*}
where $A=-Jh$ defines a symmetric $(1,2)$-tensor on the submanifold, and the Weingarten formula is given by
\begin{equation*}
D_X(JY)=J(\nabla_XY)-A(X,Y).
\end{equation*}
It is easy to see that the cubic form $C$, defined by
\begin{equation*}
C(X,Y,Z) = g(A(X,Y),Z)
\end{equation*}
is totally symmetric. For Lagrangian submanifolds of complex space forms, the equations of Gauss and Codazzi simplify to
\begin{align}
R(X,Y)Z&=\epsilon \left(X\wedge Y\right)Z+\lie{A_X,A_Y}Z,\label{Gauss1}\\
\nabla A&\; \mbox{is symmetric.}\label{Codazzi1}
\end{align}
The following theorem holds, see \cite{Chen} and \cite{CDVV}.
\begin{theorem}
Suppose $(M^n,g)$ is a Riemannian manifold equipped with a symmetric and $g$-symmetric $(1,2)$-tensor $A$ such that \eqref{Gauss1} and \eqref{Codazzi1} are satisfied for some constant $\epsilon$. Then for every point $p\in M$ there exists a neighborhood $U$ and a Lagrangian isometric immersion $\phi:U\rightarrow N^{2n}(4\epsilon)$ into the complex space form $N^{2n}(4\epsilon)$ such that $g$ and $JA$ are induced as first and second fundamental form. Such an immersion is unique up to isometries of the ambient space.
\end{theorem}

We focus on a particular form of $A$ assuming that there is a pointwise $G$-symmetry of $A$ (or equivalently of the cubic form $C$), where $G$ is a subgroup of the special orthogonal group $SO(n)$. We say that $A$ has pointwise $G$-symmetry at $p$ if for all tangent vectors $X,Y$ in $p$, and all $g\in G$ the relation $A(gX,gY)=gA(X,Y)$ holds (or equivalently $C(gX,gY,gZ)=C(X,Y,Z)$ for all $X,Y,Z$). Furthermore, we impose a minimality condition on $A$ at $p$, so for every $X$ at $p$, we assume that $\tr(A_X)=0$. These manifolds are interesting, since in $\Cx^n$ the minimal Lagrangian submanifolds are precisely the special Lagrangian submanifolds of $\Cx^n$ as introduced by Harvey and Lawson \cite{harlaw}. If a special Lagrangian submanifold of $\Cx^n$ has $G$-symmetry at every point, for the same group $G$, then a classification result for the dimension equal to $3$ has been obtained by Bryant \cite{bryant}. An explicit classification for special (we also use the word ``special'' for ``minimal'' in case $c\ne 0$) Lagrangian submanifolds of complex space forms with pointwise symmetric cubic form is not yet done, but can be easily obtained from a similar classification for affine spheres in \cite{vranck}.

In the present paper we consider the $4$-dimensional case. In particular we consider special Lagrangian 4-folds in complex space forms with pointwise symmetry. The shape of the $(1,2)$-tensor $A$, invariant under subgroups of $SO(4)$, has been described by M. Ionel in \cite{mariant}. In the same article, the author classifies special Lagrangian 4-folds of $\Cx^4$ according to their symmetry groups. However, the classification in case the symmetry group is given by $SO(2)\rtimes S_3$ in that article is incomplete; several possible subcases including the most general one is omitted. In the present article, we give a complete classification of all special Lagrangian 4-folds in any complex space form having this particular symmetry.
This settles the problem for $SO(2)\rtimes S_3$-symmetry for all $\epsilon$. The classification for other symmetry groups remains open if $\epsilon\ne0$.

The $SO(2)\rtimes S_3$-symmetry implies that $A$ can be expressed as
\begin{equation}
\begin{matrix}
A(X_1,X_1)=r X_1, & A(X_1,X_2)=-r X_2, & A(X_1,X_3)=0, & A(X_1,X_4)=0,\\
A(X_2,X_1)=-r X_2, & A(X_2,X_2)=-r X_1, & A(X_2,X_3)=0, & A(X_2,X_4)=0,\\
A(X_3,X_1)=0, &A(X_3,X_2)=0, & A(X_3,X_3)=0, &A(X_3,X_4)=0,\\
A(X_4,X_1)=0, & A(X_4,X_2)=0, & A(X_4,X_3)=0, &A(X_4,X_4)=0,\label{secondform}
\end{matrix}
\end{equation}
in a well-chosen local orthonormal frame $\{X_1,X_2,X_3,X_4\}$. In this expression $r$ is a strictly positive function. The $SO(2)$-symmetry is given by the free rotation in the $\{X_3,X_4\}$ plane and the $S_3$-symmetry is essentially obtained by rotations over an angle $2\pi/3$ in the $\{X_1,X_2\}$ plane and reflections in the $\{X_1,X_4\}$ plane. We can remark that the form of $A$ is exactly that of Lagrangian submanifolds attaining equality in Chen's inequality, see \cite{chenbook} and \cite{CDVV}.

In order to list the different possible subcases, we introduce distributions
\begin{equation*}
\begin{matrix}
\mathcal{N}_1=\spa \{X_1,X_2\}, & \mathcal{N}_+=\spa \{X_1,X_2,\lie{X_1,X_2}\}, & \mathcal{N}_2=\spa \{X_3,X_4\}.
\end{matrix}
\end{equation*}
We will see that $\mathcal{N}_2$ is always integrable. We obtain:

\begin{enumerate}
\item If $\mathcal{N}_1=\mathcal{N}_+$, then the submanifold is a double warped product $\Rl\times_f \Rl\times_g N^2$ where $N^2$ is a minimal Lagrangian submanifold in an appropriate space form.
\item If $\mathcal{N}_1\subsetneq \mathcal{N}_+$ and $\mathcal{N}_+$ is integrable, then the submanifold is a single warped product $\Rl\times_f N^3$ where $N^3$ is a special Lagrangian 3-fold with $S_3$-symmetry in an appropriate space form.
\item If the smallest integrable distribution containing $\mathcal{N}_1$ is $TM$, then for this final case, we do not obtain an explicit expression for the immersion, but we will rewrite the equations \eqref{Gauss} to a system of partial differential equations in $2$ coordinates out of $4$ coordinates defined on the submanifold. Here, techniques will be used similar to those in \cite{kriele}.
\end{enumerate}

When we consider the different cases, we will assume the defining conditions hold on an open neighborhood of the considered point.

\section{Preliminaries}
\subsection{Complex space forms.}
We briefly recall the basic properties of $\Cx^n$ and show how Lagrangian submanifolds of $\Cx P^n$ and $\Cx H^n$ can be lifted to subsets of $\Cx^{n+1}$.

Consider the complex vector space $\Cx^n$. Its elements can be written as n-tuples of complex numbers, so they are given as
\begin{equation*}
\begin{matrix}
\vec{z}=\left(z_1,\cdots,z_n\right), & z_j=x_j+i y_j, & x_j,y_j\in \Rl.
\end{matrix}
\end{equation*}
Through the map
\begin{equation*}
\phi:\Cx^n\rightarrow \Rl^{2n}:(z_1,\cdots,z_n)\rightarrow (x_1,y_1,\cdots,x_n,y_n)
\end{equation*}
the space $\Cx^n$ is a real $2n$-dimensional manifold. The multiplication with the imaginary unit $i$ translates to a linear map on $\Rl^{2n}$ given as
\begin{equation*}
i\left(x_1,y_1,\cdots,x_n,y_n\right)=\left(-y_1,x_1,\cdots,-y_n,x_n\right).
\end{equation*}
and its derivative $J$ is given as
\begin{align*}
J\partial_{x_k}&=\partial_{y_k},\\
J\partial_{y_k}&=-\partial_{x_k}.
\end{align*}
This squares to $-I$ and thus defines a complex structure on $\Cx^n$. On $\Cx^n$ there is also a Hermitian form given by
\begin{equation*}
s(\vec{z},\vec{w})=\sum_{j=1}^n z_j \bar{w}_j
                =\sum_{j=1}^n(x_ju_j+y_jv_j)-i \sum_{j=1}^n (x_j v_j-y_j u_j).
\end{equation*}
The real part, which can be denoted as $\scal{\vec{z},\vec{w}}$ defines the Euclidean scalar product on $\Rl^{2n}$ and induces a natural Riemannian metric on $\Cx^n$. We can see that $J$ is an isometry and the induced K\"ahler form, which also coincides with the imaginary part of the Hermitian form, is closed. These structures make $\Cx^n$ into a flat K\"ahler manifold.

The manifold $\Cx P^n$ can be modeled as the quotient $S^{2n+1}/S^1$, where
\begin{equation*}
S^{2n+1}=\{(z_0,\cdots,z_n)\in \Cx^{n+1}|\sum_{i=0}^n\abs{z_i}^2=1\}.
\end{equation*}
The equivalence is given by
\begin{equation*}
\vec{z}\sim\vec{w} \Leftrightarrow \exists \phi\in \Rl \;\forall j\in\{0,\cdots,n\}: z_j=e^{i\phi}w_j.
\end{equation*}
So the unit sphere $S^{2n+1}$ is the preimage of the Hopf fibration
\begin{equation*}
\pi:S^{2n+1}\rightarrow\Cx P^n: \vec{z}\rightarrow\lie{\vec{z}}.
\end{equation*}
On $S^{2n+1}\subset\Cx^{n+1}$ the complex structure $J$ induces a contact structure and the standard metric on $\Cx^{n+1}$ induces a Riemannian metric. The metric on $\Cx P^n$ that makes $\pi$ a Riemannian submersion has constant holomorphic sectional curvature $4$. An immersion $\phi:M \rightarrow S^{2n+1}$ is then said to be C-totally real or horizontal if $i\phi$ is orthogonal to the submanifold. It can be shown that every minimal C-totally real submanifold of $S^{2n+1}$ can be projected onto a special Lagrangian submanifold of $\Cx P^n$ through $\pi$ and conversely that a special Lagrangian submanifold in $\Cx P^n$ has a 1-parameter family of mutually isometric horizontal lifts as a minimal C-totally real submanifold in $S^{2n+1}$. So in order to classify special Lagrangian submanifolds in $\Cx P^n$, we can consider minimal C-totally real submanifolds in $S^{2n+1}\subset \Cx^{n+1}$, see \cite{ziegel}. For those submanifolds, the Gauss identity is given as
\begin{equation}
D_XY=\nabla_XY+JA(X,Y)-\scal{X,Y} \phi,\label{Gaussproj}
\end{equation}
where $D$ is the Levi Civita connection of $\Cx^{n+1}$.

Similarly, the space $\Cx H^n$ can be modeled as $H^{2n+1}/S^1$, where
\begin{equation*}
H^{2n+1}=\{(z_0,\cdots,z_n)\in \Cx^{n+1}_1|\abs{z_0}^2-\sum_{i=1}^{n} \abs{z_i}^2=1\}.
\end{equation*}
The equivalence relationship determined by $S^1$ is the same as the one used in the projective space. The ambient space $\Cx^{n+1}_1$ is essentially the space $\Cx^{n+1}$, but equipped with the scalar product
\begin{equation*}
\scal{\vec{z},\vec{w}}_1=\Re\left(\sum_{j=1}^n z_j\bar{w}_j-z_0\bar{w}_0\right).
\end{equation*}
The complex structure is still obtained through multiplication with the imaginary unit $i$ and induces a K\"ahler structure on $\Cx^{n+1}_1$. This metric induces a Lorentzian metric on $H^{2n+1}$ and a metric of constant holomorphic sectional curvature $-4$ on $\Cx H^n$. Similar to the projective case C-totally real submanifolds $\phi:M\rightarrow H^{2n+1}$ can be defined having $i\phi$ as a normal. Each minimal C-totally real submanifold corresponds to the horizontal lift of a  special Lagrangian submanifold of $\Cx H^n$. The Gauss identity is given as
\begin{equation}
D_XY=\nabla_XY+JA(X,Y)+\scal{X,Y} \phi,\label{Gausshyp}
\end{equation}
where $D$ is the Levi Civita connection of $\Cx^{n+1}_1$.

\subsection{Structure equations.}
We can return briefly to the equations \eqref{Gauss1} and \eqref{Codazzi1}. We can choose an orthogonal frame $\{X_1,X_2,X_3,X_4\}$ corresponding to \eqref{secondform} and define the components $\Gamma_{ij}^k$ and $A_{ij}^k$ as
\begin{align*}
\nabla_{X_i}X_j&=\sum_{k=1}^4 \Gamma_{ij}^k X_k,\\
A(X_i,X_j)&=\sum_{k=1}^4 A_{ij}^k X_k.
\end{align*}
Then the equations \eqref{Gauss1} and \eqref{Codazzi1} can be rewritten as
\begin{align}
\begin{split}\label{Gauss}
X_i\left(\Gamma_{jk}^l\right)-X_j\left(\Gamma_{ik}^l\right)=&\epsilon \left(\delta_{jk}\delta_{i}^l-\delta_{ik}\delta_{j}^l\right)+A_{jk}^rA_{ir}^l-A_{ik}^rA_{jr}^l\\
&+\Gamma_{ik}^r\Gamma_{jr}^l-\Gamma_{jk}^r\Gamma_{ir}^l+\Gamma_{rk}^l\left(\Gamma_{ij}^r-\Gamma_{ji}^r\right),
\end{split}\\
X_i\left(A_{jk}^l\right)-X_j\left(A_{ik}^l\right)=&\left(\Gamma_{ij}^r-\Gamma_{ji}^r\right)A_{rk}^l+\Gamma_{ik}^rA_{jr}^l-\Gamma_{jk}^rA_{ir}^l-\Gamma_{ir}^lA_{jk}^r+\Gamma_{jr}^lA_{ik}^r,\label{Codazzi}
\end{align}
where we have used the Einstein convention. We split the connection $\nabla$ into its components and write
\begin{align*}
\begin{matrix}
\nabla_{X_1}X_1=a_1X_2+a_2X_3+a_3X_4, &\vline & \nabla_{X_1}X_2=-a_1X_1+a_4X_3+a_5X_4,\\
\nabla_{X_2}X_1=b_1X_2+b_2X_3+b_3X_4, &\vline & \nabla_{X_2}X_2=-b_1X_1+b_4X_3+b_5X_4,\\
\nabla_{X_3}X_1=c_1X_2+c_2X_3+c_3X_4, &\vline & \nabla_{X_3}X_2=-c_1X_1+c_4X_3+c_5X_4,\\
\nabla_{X_4}X_1=d_1X_2+d_2X_3+d_3X_4, &\vline & \nabla_{X_4}X_2=-d_1X_1+d_4X_3+d_5X_4,\\
\end{matrix}\\
\begin{matrix}
\nabla_{X_1}X_3=-a_2X_1-a_4X_2+a_6X_4, &\vline & \nabla_{X_1}X_4=-a_3X_1-a_5X_2-a_6X_3,\\
\nabla_{X_2}X_3=-b_2X_1-b_4X_2+b_6X_4, &\vline & \nabla_{X_2}X_4=-b_3X_1-b_5X_2-b_6X_3,\\
\nabla_{X_3}X_3=-c_2X_1-c_4X_2+c_6X_4, &\vline & \nabla_{X_3}X_4=-c_3X_1-c_5X_2-c_6X_3,\\
\nabla_{X_4}X_3=-d_2X_1-d_4X_2+d_6X_4, &\vline & \nabla_{X_4}X_4=-d_3X_1-d_5X_2-d_6X_3.\\
\end{matrix}
\end{align*}
Equation \eqref{Codazzi} induces linear relations between the components, independent of the ambient space. The Gauss equations give further information about $\nabla$ but use differential equations and depend on the ambient space form.
\begin{lemma}
On a special Lagrangian submanifold $M$ having a local $SO(2)\rtimes S_3$-symmetry there exists a frame corresponding to \eqref{secondform} such that:
\begin{equation}
\begin{matrix}
&\nabla_{X_1}X_1=a_1X_2+a_2X_3+a_3X_4, &\vline & \nabla_{X_1}X_2=-a_1X_1-b_2X_3,\\
&\nabla_{X_2}X_1=b_1X_2+b_2X_3, &\vline & \nabla_{X_2}X_2=-b_1X_1+a_2X_3+a_3X_4,\\
&\nabla_{X_3}X_1=\frac{b_2}{3}X_2, &\vline & \nabla_{X_3}X_2=-\frac{b_2}{3}X_1,\\
&\nabla_{X_4}X_1=0, &\vline & \nabla_{X_4}X_2=0,\\
&\nabla_{X_1}X_3=-a_2X_1+b_2X_2+a_6X_4, &\vline & \nabla_{X_1}X_4=-a_3X_1-a_6X_3,\\
&\nabla_{X_2}X_3=-b_2X_1-a_2X_2+b_6X_4, &\vline & \nabla_{X_2}X_4=-a_3X_2-b_6X_3,\\
&\nabla_{X_3}X_3=c_6X_4, &\vline & \nabla_{X_3}X_4=-c_6X_3,\\
&\nabla_{X_4}X_3=d_6X_4, &\vline & \nabla_{X_4}X_4=-d_6X_3.\\
\end{matrix}\label{Conmat}
\end{equation}
Furthermore, the derivatives of $r$ are given by
\begin{align}
(X_1+i X_2)(r)&=3ir(a_1+ib_1),\label{cod1r}\\
X_3(r)&=ra_2,\label{cod3r}\\
X_4(r)&=r a_3.\label{cod4r}
\end{align}
\end{lemma}
\begin{proof}
This is just a straightforward application of equation \eqref{Codazzi}. For instance
\begin{align*}
\left(\nabla_{X_2}A\right)(X_1,X_1)&=X_2(r)X_1+3rb_1X_2+rb_2X_3+rb_3X_4,\\
\left(\nabla_{X_1}A\right)(X_2,X_1)&=3ra_1X_1-X_1(r) X_2-ra_4X_3-ra_5X_4.
\end{align*}
Then the corresponding coordinates of both derivatives are the same. Finally we can set $b_3=0$, by rotating the distribution $\mathcal{N}_2$ such that $X_3$ lies in the direction of $\nabla_{X_1} X_2$, projected on $\mathcal{N}_2$.
\end{proof}
It is interesting to note that $\mathcal{N}_2$ is an integrable distribution. The distribution $\mathcal{N}_1$ however is integrable if and only if $b_2=0$. Applying \eqref{Gauss}, we obtain the following result.
\begin{lemma}
The equations \eqref{Gauss} on our frame of choice induce a system of differential equations given by:
\begin{align}
(X_1+iX_2)(a_2-ib_2)&=a_3(a_6+ib_6), \label{gaua122}\\
X_3(a_2+i b_2)&=\epsilon+a_3^2+(a_2+ib_2)^2, \label{gaua32}\\
X_4(a_2+i b_2)&=a_3 (a_2+i b_2), \label{gaua42}\\
(X_1+iX_2)(a_3)&=-(a_2-i b_2)(a_6+i b_6), \label{gaua123}\\
X_3(a_3)&=0, \label{gaua33}\\
X_4(a_3)&=a_3^2+\epsilon, \label{gaua43}\\
X_1(b_6)-X_2(a_6)&=-(a_1a_6+b_1b_6), \label{gaua126-}\\
X_3(a_6+i b_6)&=\frac{5}{3} i b_2 (a_6+ib_6), \label{gaua36}\\
X_4(a_6+i b_6)&=2 a_3 (a_6+i b_6), \label{gaua46}\\
X_1(b_1)-X_2(a_1)&=2r^2-(\epsilon+a_3^2)-\frac{5}{3}b_2^2-a_2^2-a_1^2-b_1^2, \label{gaua121-}\\
X_2(b_1)+X_1(a_1)&=-\frac{2}{3}a_2b_2, \label{gaua121+}\\
3X_3(a_1)-X_1(b_2)&=3a_1a_2-2b_1b_2, \label{gaua131}\\
3X_3(b_1)-X_2(b_2)&=2b_2a_1+3b_1a_2, \label{gaua231}\\
X_4(a_1+ib_1)&=a_3(a_1+ib_1)+\frac{b_2}{3}(a_6+i b_6). \label{gaua41}
\end{align}
\end{lemma}

\begin{proof}
This is also a straightforward application of \eqref{Gauss}. For example:
\begin{align*}
X_1\left(\Gamma_{23}^1\right)-X_2\left(\Gamma_{13}^1\right)&= \Gamma_{13}^r\Gamma_{2r}^1-\Gamma_{23}^r\Gamma_{1r}^1+\Gamma_{r3}^1\Gamma_{12}^r-\Gamma_{r3}^1\Gamma_{21}^r\\
&=a_3 b_6,\\
X_1\left(\Gamma_{23}^2\right)-X_2\left(\Gamma_{13}^2\right)&=\Gamma_{13}^r\Gamma_{2r}^2-\Gamma_{23}^r\Gamma_{1r}^2+\Gamma_{r3}^2\Gamma_{12}^r-\Gamma_{r3}^2\Gamma_{21}^r\\
&=-a_3 a_6.
\end{align*}
Combining both equations using the usual complex notations leads to \eqref{gaua122}. The other equations are obtained in a similar way.
\end{proof}

\subsection{Warped Products.}
In the analysis that follows, we will often encounter warped products of manifolds. When we consider a warped product of Riemannian manifolds $(M_1,g_1)$ and $(M_2,g_2)$ with warping function $e^f$, where $f:M_1\rightarrow \Rl$, we get a Riemannian manifold $(M_1\times M_2,g_f)$ where $M_1\times M_2$ as a differentiable manifolds is the product of $M_1$ and $M_2$ and the metric $g_f$ is given as
\begin{equation*}
g_f(X,Y)=g_1(X_1,Y_1)+e^{2f} g_2(X_2,Y_2),
\end{equation*}
where a vector field $X$ is uniquely decomposed into a part $X_1$ tangent to $M_1$ and $X_2$ tangent to $M_2$. We denote this warped product as $M_1\times_{e^f} M_2$. The following result can be obtained, see \cite{hiepko}.
\begin{theorem}
Consider a Riemannian manifold $(M,g)$ with Levi-Civita connection $\nabla$ and suppose that on a neighborhood of $p\in M$ there are orthogonal distributions $\mathcal{N}_1$ and $\mathcal{N}_2$ such that
\begin{align*}
&\forall X,Y\in \mathcal{N}_1 \text{(i.e. $X$ and $Y$ are sections of $\mathcal{N}_1$)}  : \nabla_XY\in \mathcal{N}_1,\\
&\forall X,Y\in \mathcal{N}_2  : \nabla_XY=\tilde{\nabla}_XY+g(X,Y) H,
\end{align*}
where $\tilde{\nabla}$ is the projection of $\nabla$ on $\mathcal{N}_2$ and $H\in\mathcal{N}_1$. Then there exists a function $f:M\rightarrow \Rl$ such that on a neighborhood of $p$, $M$ can be written as $M_1\times_{e^f} M_2$, where $M_i$ is an integral manifold of $\mathcal{N}_i$.\\
If furthermore $H=\lambda H_0$, where $\norm{H_0}=1$, and $X(\lambda)=0$ for every $X\in\mathcal{N}_2$, then
$f:M_1\rightarrow \Rl$ and $H=-\grad{e^f}$.
\end{theorem}
The first part of the theorem constructs a twisted product, the second part reduces this to a warped product.
This will be useful in choosing coordinates, since the product structures allows for coordinates to be chosen on each factor separately. In particular, if $\dim\left(\mathcal{N}_1\right)=1$, then any non vanishing vector field in $\mathcal{N}_1$ can be fixed as a useful coordinate vector field on $M$.\\

\section{Submanifolds in $\Cx^4$.}
\subsection{The case where $b_2=0$.}
The assumption that $X_3$ lies along $\nabla_{X_1} X_2$ becomes redundant since the latter has no $\mathcal{N}_2$ component. Instead, we can choose $X_3$ in the direction of $\nabla_{X_1} X_1$, projected on $\mathcal{N}_2$. Hence without loss of generality we can assume that $a_3=0$. The equations \eqref{Gauss} show that in this case either $a_2=0$ or $a_6=b_6=c_6=d_6=0$. First we will assume that $a_2\neq0$.
\begin{theorem}
Consider $M$ a special Lagrangian submanifold in $\Cx^4$ having $SO(2)\rtimes S_3$-symmetry and an orthogonal frame corresponding to \eqref{secondform}. Suppose that $\mathcal{N}_1$ is an integrable distribution and $\nabla_{X_1} X_1$  is nowhere contained within this distribution. Then $M$ is locally congruent to
\begin{equation}
F(t,s,u,v)=(t,s \phi(u,v))
\end{equation}
where $\phi$ is the horizontal lift of a special Lagrangian submanifold of $\Cx P^2$ to the unit sphere in $\Cx^3$.
\end{theorem}
\begin{proof}
Taking into account every component that vanishes in \eqref{Conmat}, we find
\begin{align*}
\begin{matrix}
&\nabla_{X_1}X_1=a_1X_2+a_2X_3&\vline & \nabla_{X_1}X_2=-a_1X_1,\\
&\nabla_{X_2}X_1=b_1X_2 &\vline & \nabla_{X_2}X_2=-b_1X_1+a_2X_3,\\
&\nabla_{X_3}X_1=0 &\vline & \nabla_{X_3}X_2=0,\\
&\nabla_{X_4}X_1=0 &\vline & \nabla_{X_4}X_2=0,\\
&\nabla_{X_1}X_3=-a_2X_1 &\vline & \nabla_{X_1}X_4=0,\\
&\nabla_{X_2}X_3=-a_2X_2&\vline & \nabla_{X_2}X_4=0,\\
&\nabla_{X_3}X_3=0 &\vline & \nabla_{X_3}X_4=0,\\
&\nabla_{X_4}X_3=0 &\vline & \nabla_{X_4}X_4=0.\\
\end{matrix}
\end{align*}
We find that the distributions $\spa\{X_4\}$ and $\spa\{X_1,X_2,X_3\}$ satisfy the conditions for a warped product $\Rl\times_{e^f} N^3$. But $X_4(f)=0$, hence $f$ is a constant. $M$ is a standard Riemannian product $\Rl\times N^3$ and its immersion can be written, up to an isometry as
\begin{equation*}
F(t,x)=\left(t,\psi(x)\right), \,
\psi:N^3\rightarrow \Cx^4.
\end{equation*}
The immersion $\psi$ is contained in the subspace orthogonal to both $X_4$ and $J X_4$, since they both are constant unit normals along $N^3$. Now it is also obvious that $\spa\{X_3\}$ and $\mathcal{N}_1$ satisfy the conditions for a warped product. So $N^3$ can be decomposed as $\Rl\times_{e^g} N^2$ and $X_3(g)=-a_2$. Then $X_3$ can be associated with a coordinate $s$ on the manifold and it follows that
\begin{equation*}
D_{X_3}X_3=\pardd{F}{s}=0 \,
          \Rightarrow F=A s+B.
\end{equation*}
Both $A$ and $B$ are independent of $(s,t)$. Calculating \eqref{Gauss}, one has
\begin{equation*}
X_3(a_2)=\pard{a_2}{s}=a_2^2.
\end{equation*}
The solution of this equation, after a translation of the $s$-coordinate is given as $a_2=-\frac{1}{s}$. The derivatives of $X_3$ to $X_1$ and $X_2$ are
\begin{equation*}
\begin{split}
D_{X_i} X_3&= \pard{F_* X_i}{s}= A_* X_i\\
           &=\frac{1}{s} X_i=A_*X_i+\frac{B_* X_i}{s}\\
           &\Rightarrow  B_*=0.
\end{split}
\end{equation*}
So $B$ is a constant vector along the submanifold and vanishes when applying a translation. It is easy to see that $X_3=A$ and is orthogonal to $X_i=s A_*(X_i)$, for $i\in \{1,2\}$. Hence everywhere along $A$, the position vector is orthogonal to the tangent space. Thus $A$ has constant length. Calculating the other covariant derivatives yields for example for $i,j \in\{1,2\}$ that
\begin{equation}
\begin{split}
A_*X_i&=\frac{F_*X_i}{s},\\
D_{X_i}\left(A_*X_j\right)&=\frac{D_{X_i} \left(F_*X_j\right)}{s}=A_*\left(\tilde{\nabla}_{X_i} X_j\right)+J A_*\left( K(X_i,X_j)\right)-\frac{1}{s^2}\delta_{ij} \phi. \label{refren}
\end{split}
\end{equation}
Here $\tilde{\nabla}$ is the connection restricted to $N^2$. Combining this with the other equations in \eqref{Gauss}, it follows that $A$ is a C-totally real immersion in $S^5\subset \Cx^3$. Furthermore, the components $a_1$ and $b_1$ have no other restrictions on them except satisfying the Gauss equations for a minimal C-totally real submanifold of $S^5$.  This proves the theorem.
\end{proof}
The case $a_2=0$ was the only case that was studied in \cite{mariant}. We can quote the following result from \cite{mariant}.
\begin{theorem}
Consider $M$ a special Lagrangian submanifold in $\Cx^4$ having a local $SO(2)\rtimes S_3$-symmetry group and an orthogonal frame corresponding to \eqref{secondform}. Suppose that $\mathcal{N}_1$ is an integrable distribution and $\nabla_{X_1} X_1$  is contained within this distribution. Then $M$ is locally congruent to
\begin{equation}
F(t,s,u,v)=\left(t,s,\phi(u,v)\right)
\end{equation}
where $\phi:\Cx\rightarrow \Cx^2$ is a special Lagrangian surface.
\end{theorem}

\begin{remark}
As proved in \cite{mariant}, a special Lagrangian surface in $\Cx^2$, with complex coordinates $x_1+iy_1$ and $x_2+iy_2$ is a holomorphic curve in $\Cx^2$ with complex coordinates $x_1-ix_2$ and $y_1+iy_2$, and conversely.
\end{remark}

\subsection{The case where $b_2\neq 0$}
Now the distribution $\mathcal{N}_1$ is no longer integrable. The simplest case one can hope for is that there is a $3$-dimensional integrable distribution containing $\mathcal{N}_1$. Such a distribution should contain at least $X_3$ since
$$\lie{X_1,X_2}\mod \mathcal{N}_1 \parallel X_3.$$
Using the fact that $b_2 \neq 0$, the equations \eqref{Gauss} reduce \eqref{Conmat} to
\begin{equation}
\begin{matrix}
&\nabla_{X_1}X_1=a_1X_2+a_2X_3+a_3X_4 &\vline & \nabla_{X_1}X_2=-a_1X_1-b_2X_3,\\
&\nabla_{X_2}X_1=b_1X_2+b_2X_3 &\vline & \nabla_{X_2}X_2=-b_1X_1+a_2X_3+a_3X_4,\\
&\nabla_{X_3}X_1=\frac{b_2}{3}X_2 &\vline & \nabla_{X_3}X_2=-\frac{b_2}{3}X_1,\\
&\nabla_{X_4}X_1=0 &\vline & \nabla_{X_4}X_2=0,\\
&\nabla_{X_1}X_3=-a_2X_1+b_2X_2+a_6X_4 &\vline & \nabla_{X_1}X_4=-a_3X_1-a_6X_3,\\
&\nabla_{X_2}X_3=-b_2X_1-a_2X_2+b_6X_4 &\vline & \nabla_{X_2}X_4=-a_3X_2-b_6X_3,\\
&\nabla_{X_3}X_3=a_3X_4 &\vline & \nabla_{X_3}X_4=-a_3X_3,\\
&\nabla_{X_4}X_3=0 &\vline & \nabla_{X_4}X_4=0.\\
\end{matrix}\label{Conmat2}
\end{equation}
It is apparent that the condition that $\mathcal{N}_+$ is integrable is given by $a_6+ib_6=0$. We consider this case first.
\begin{theorem}
Suppose $M$ is a special Lagrangian submanifold in $\Cx^4$ with  $SO(2)\rtimes S_3$-symmetry, such that $\mathcal{N}_1$ is not an integrable distribution, but $\mathcal{N}_+$ is. Then the submanifold, up to isometry, can be given locally by either
\begin{equation}
F(t,s,u,v)=\left(t,\phi(s,u,v)\right),
\end{equation}
where $\phi$ is a special Lagrangian submanifold with $S_3$-symmetry in $\Cx^3$, or
\begin{equation}
F(t,s,u,v)=t \phi(s,u,v)
\end{equation}
where $\phi$ is the horizontal lift of a special Lagrangian submanifold with local $S_3$-symmetry in $\Cx P^3$ to the unit sphere in $\Cx^4$.
\end{theorem}
\begin{proof}
We find according to \eqref{Conmat2} and \eqref{Gauss} that $\spa\{X_4\}$ and $\mathcal{N}_+$ satisfy the conditions for a warped product. So $M$ can be decomposed as $\Rl \times_{e^f} N^3$, where $X_4(f)=-a_3$. We can solve
\begin{equation*}
X_4(a_3)=\pard{a_3}{t}=a_3^2.
\end{equation*}
This equation has $2$ possible solutions.\\
First, we assume $a_3=0$. In this case $M$ is simply the manifold $\Rl\times N^3$. Hence the immersion, up to isometry, can be given as
\begin{equation*}
F(t,s,u,v)=(t,\phi(s,u,v)),
\end{equation*}
where $\phi$ is a $3$-fold immersed in the subspace $\Cx^3$ orthogonal to $X_4$ and $J X_4$. Similar calculations as in \eqref{refren} show that this can be any special Lagrangian submanifold in $\Cx^3$, given the presence of an $S_3$-symmetry in the second fundamental form.\\
The second solution, after a translation of $t$, is given by $a_3=-\frac{1}{t}$. The calculations are similar to the case where $b_2=0$ and $a_2\neq 0$. This gives the required result.
\end{proof}

The last case in $\Cx^4$ is the one where there is no integrable distribution containing $\mathcal{N}_1$ other than the whole tangent bundle. In this case, we can no longer rely on an obvious warped product structure. We can attempt to introduce a set of independent coordinates and reduce \eqref{Gauss} to a system of PDE's on $\Cx^4$ using as little functions as possible. We now use \eqref{gaua122} to \eqref{gaua41} to construct a coordinate frame from $\{X_1,X_2,X_3,X_4\}$. Since $\mathcal{N}_2$ is integrable, it is a good idea to choose $X_4=T$ and $\mu X_3=S$. Requiring that $\lie{S,T}=0$ implies that
\begin{equation*}
\lie{\mu X_3,X_4}=\mu \lie{X_3,X_4}-X_4(\mu) X_3=-\left(\mu a_3+X_4(\mu)\right) X_3=0.
\end{equation*}
We can find such a $\mu$ by taking $\mu=\frac{1}{\sqrt{\abs{\epsilon+a_3^2}}}$.  The equation $a_3^2+\epsilon=0$ implies that $a_3$ is a constant and hence $(a_2-ib_2)(a_6+ib_6)=0$. This will correspond to the integrability of either $\mathcal{N}_1$ or $\mathcal{N}_+$. Therefore $\mu$ is well defined.

Vector fields $U$ and $V$ can be sought such that every couple out of $\{S,T,U,V\}$ commutes. Such an attempt can be made, writing
\begin{equation}
U+iV=(\rho_1-i\rho_2)\Big( (X_1+iX_2)+(\alpha_1+i\beta_1) S+(\alpha_2+i\beta_2)T \Big)\label{coorvf}
\end{equation}
We rename the following expressions:
\begin{align*}
\rho&=\rho_1-i\rho_2,\\
\gamma_j&=\alpha_j+i \beta_j\qquad j\in \{1,2\}.
\end{align*}
After calculating the Lie brackets of these four vector fields, the following conditions on the introduced functions make the vector fields commute:
\begin{align}
X_4(\rho) &=-a_3 \rho,\label{coor4r}\\
X_3(\rho) &=-\left(a_2+\frac{2}{3}i  b_2\right) \rho, \label{coor3r}\\
(X_1-i X_2)(\rho) &=\left(b_1+ia_1\right)\rho, \label{coor12r}\\
X_4(\gamma_1) &=-\frac{1}{\mu} (a_6+i b_6)+a_3 \gamma_1, \label{coor4g1}\\
X_3(\gamma_1) &=\frac{1}{\mu^2}(X_1+iX_2)(\mu)+\gamma_1\left(a_2+\frac{2}{3}i  b_2\right), \label{coor3g1}\\
X_2(\alpha_1)-X_1(\beta_1) &=a_1 \alpha_1+b_1 \beta_1-\frac{2}{\mu} b_2 , \label{coor12g1}\\
X_4(\gamma_2) &=a_3 \gamma_2, \label{coor4g2}\\
X_3(\gamma_2) &=(a_6+i b_6)+\left(a_2+\frac{2}{3}i b_2\right) \gamma_2, \label{coor3g2}\\
X_2(\alpha_2)-X_1(\beta_2) &=a_1 \alpha_2+b_1 \beta_2\label{coor12g2}.
\end{align}
The following result can be obtained.
\begin{lemma}
Suppose $f$ and $g$ are real valued functions on the manifold satisfying
\begin{equation*}
\begin{matrix}
S(f)=0, & T(f)=-1,\\ S(g)=-1, & T(g)=0,
\end{matrix}
\end{equation*}
and defining
\begin{equation*}
\begin{matrix}
X_1(f)=\alpha_2, & X_2(f)=\beta_2,\\ X_1(g)=\alpha_1, & X_2(g)=\beta_1,
\end{matrix}
\end{equation*}
then the functions $\alpha_i$ and $\beta_i$ obtained this way satisfy the conditions \eqref{coor4g1} to \eqref{coor12g2}.
\end{lemma}
It is interesting to see that this way the vector fields
\begin{align*}
\tilde{U}&= X_1+\alpha_1S+\alpha_2T,\\
\tilde{V}&= X_2+\beta_1 S+\beta_2 T,
\end{align*}
satisfy $\tilde{U}(f)=\tilde{U}(g)=\tilde{V}(f)=\tilde{V}(g)=0$. Furthermore $\tilde{U}$ and $\tilde{V}$ are independent of one-another and they span the distribution which is the intersection of the kernel of $\dv f$ and $\dv g$. Note that this distribution is indeed $2$-dimensional since both forms have a hyperplane as a kernel and these kernels do not coincide, since the $1$-forms are linearly independent. Using the dimension theorem, they have a $2$-dimensional intersection. Construction \eqref{coorvf} is just a complex rotation of these two vector fields in that distribution. This way, it is clear that $f$ and $g$ serve as coordinates $s$ and $t$ conjugate to $S$ and $T$.
\begin{proof}
Apply the relation
\begin{equation*}
\lie{X_i,X_j}(f)=X_iX_j(f)-X_jX_i(f)=\nabla_{X_i}X_j(f)-\nabla_{X_j}X_i(f)
\end{equation*}
on both functions, using \eqref{Conmat2}.
\end{proof}
A suitable function for $f$ is easily found, since $S(a_3)=0$. Let $f$ be a function of $a_3$, then
\begin{equation*}
X_4(f)=f' (\epsilon+a_3^2)=-1 \Leftrightarrow f'=-\frac{1}{\epsilon+a_3^2}.
\end{equation*}
Hence $f$ can be given by
\begin{equation*}
f=-\int{\frac{1}{\epsilon+a_3^2}\dv a_3}.\label{ffun}
\end{equation*}
This also determines $\gamma_2$ completely, since using \eqref{gaua122} yields
\begin{align*}
\gamma_2&=(X_1+iX_2)(f)=f' (X_1+iX_2)(a_3)\\
        &=\frac{a_2a_6+b_2b_6}{\epsilon+a_3^2}+i\frac{a_2b_6-b_2a_6}{\epsilon+a_3^2}.
\end{align*}
As for the function $g$, the complex valued function $z=\mu (a_2+i b_2)$ can be considered and calculations show
\begin{align*}
X_4(z)&=-\mu a_3 (a_2+i b_2)+\mu a_3(a_2+i b_2)=0,\\
S(z)&=\mu^2 \left(\epsilon+a_3^2+(a_2+i b_2)^2 \right)=\sign\left(\epsilon+a_3^2\right)+z^2.
\end{align*}
Rewriting $\tilde{\epsilon}=\sign(\epsilon+a_3^2)$, we find that $z$ is useful as long as $z^2+\tilde{\epsilon}\neq 0$. When $\tilde{\epsilon}=+1$, this occurs when $a_2=0$ and $\abs{b_2}=\sqrt{\epsilon+a_3^2}$. When $\tilde{\epsilon}=-1$, this occurs when $\abs{a_2}=\sqrt{\abs{\epsilon+a_3^2}}$ and $b_2=0$, resulting in $\mathcal{N}_1$ being integrable.

 First we assume that $z^2\neq -\tilde{\epsilon}$. Then the function $g$ can be calculated as the real part of a function $G$ of $z$ given by
\begin{equation*}
S(G)=(\tilde{\epsilon}+z^2) G'=-1 \Leftrightarrow G'=-\frac{1}{\tilde{\epsilon}+z^2}.
\end{equation*}
A function $\rho$ still has to be constructed satisfying \eqref{coor4r} to \eqref{coor12r}. Define a function $H$ as
\begin{equation*}
H=\rho^3 r (z^2+\tilde{\epsilon})\, \abs{\epsilon+a_3^2}.
\end{equation*}
This function is a constant on the submanifold and can be taken to be equal to $1$. This defines a function $\rho$ satisfying the necessary conditions.

Using the Frobenius theorem in \cite{nomkob1}, a coordinate frame on the submanifold is given by
\begin{align*}
X_4&=T,\\
X_3&= \frac{1}{\mu} S,\\
X_1+iX_2&= \frac{U+iV}{\rho}-\gamma_1 S-\gamma_2 T.
\end{align*}
We can describe the dependence of $a_6+i b_6$ on $(s,t)$ by writing
\begin{equation*}
a_6+i b_6=\frac{k_3+ik_4}{\rho} \sqrt{\abs{a_3^2+\epsilon}} \left(\bar{z}^2+\tilde{\epsilon}\right)^{\frac{-1}{2}}.
\end{equation*}
The functions $k_3$ and $k_4$ depend solely on $(u,v)$. This expression is obtained from \eqref{gaua36} and \eqref{gaua46}.
The rest of the equations in \eqref{Gauss} can be rewritten and solved. Applying our method for $\epsilon=0$, we find after a translation of the coordinates that
\begin{align*}
a_3 &= -\frac{1}{t},\\
x=\frac{\sin(2 s)}{\cos(2 s)+\cosh(2 k_1)}&\Rightarrow a_2=-\frac{\sin(2 s)}{t(\cos(2 s)+\cosh(2 k_1))},\\
y=\frac{\sinh(2 k_1)}{\cos(2 s)+\cosh(2 k_1)} &\Rightarrow b_2=-\frac{\sinh(2 k_1)}{t(\cos(2 s)+\cosh(2 k_1))},\\
r&=\frac{e^{k_2}}{t \sqrt{\cos(2s)+\cosh(2 k_1)}}.
\end{align*}
Here the functions $k_1$ and $k_2$ depend solely on $(u,v)$. Then we can use \eqref{gaua122} to find an expression for $\gamma_1$ in terms of the coordinates. Equation \eqref{cod1r} can be used to find an expression for $a_1$ and $b_1$ in terms of the coordinates. We obtain
\begin{align*}
\gamma_1&=\frac{(k_3+ik_4)\cos(s-ik_1)+t\left(\pard{k_1}{v}-i\pard{k_1}{u}\right)}{t\rho},\\
a_1&=\frac{2^{\frac{2}{3}}e^{\frac{2}{3}k_2}}{3t^3\left(\cos(2s)+\cosh(2k_1)\right)^2}\big(t\big((\cos(2s)+\cosh(2k_1))(\rho_1\pard{k_2}{v}+\rho_2\pard{k_2}{u})\\
&\qquad+\sin(2s)(\rho_1\pard{k_1}{u}-\rho_2\pard{k_1}{v})-\sinh(2k_1)(\rho_1\pard{k_1}{v}+\rho_2\pard{k_1}{u})\big)\\
&\qquad + \sinh(2k_1) \left(\cos(s)\cosh(k_1)(k_4\rho_2-k_3\rho_1)+\sin(s)\sinh(k_1)(k_4\rho_1+k_3\rho_2)\right)\big),\\
b_1&=\frac{2^{\frac{2}{3}}e^{\frac{2}{3}k_2}}{3t^3\left(\cos(2s)+\cosh(2k_1)\right)^2}\big(t\big((\cos(2s)+\cosh(2k_1))(\rho_2\pard{k_2}{v}-\rho_1\pard{k_2}{u})\\
&\qquad+\sin(2s)(\rho_2\pard{k_1}{u}+\rho_1\pard{k_1}{v})+\sinh(2k_1)(\rho_1\pard{k_1}{u}-\rho_2\pard{k_1}{v})\big)\\
&\qquad + \sinh(2k_1) \left(\sin(s)\sinh(k_1)(k_4\rho_2-k_3\rho_1)-\cos(s)\cosh(k_1)(k_4\rho_1+k_3\rho_2)\right)\big).
\end{align*}

Now every function on the submanifold is expressed in terms of $(s,t,u,v)$, possibly indirectly through $\{k_1,k_2,k_3,k_4\}$. Demanding that the other Gauss equations are satisfied gives partial differential equations for $k_i$, given by
\begin{equation}
\begin{split}
\frac{\partial k_4}{\partial u}-\frac{\partial k_3}{\partial v}&=2 \tanh(k_1) \left(k_3 \frac{\partial k_1}{\partial v}-k_4\frac{\partial k_1}{\partial u}\right),\\
\frac{\partial k_4}{\partial v}+\frac{\partial k_3}{\partial u}&=-2 \coth(k_1) \left(k_3 \frac{\partial k_1}{\partial u}+k_4\frac{\partial k_1}{\partial v}\right),\\
\Delta k_2&=3*2^{\frac{1}{3}}e^{-\frac{2 k_2}{3}}\left(-e^{2k_2}+\cosh(2k_1)\right),\\
\Delta k_1&=-2^{\frac{1}{3}}e^{-\frac{2k_2}{3}} \sinh(2k_1).\label{constraa}
\end{split}
\end{equation}
Now we return to the case where $-1=z^2$ and $\tilde{\epsilon}=1$. We assume first that $\epsilon$ isn't specified. In this case $a_2=0$,  $b_2=\pm\sqrt{\epsilon+a_3^2}$ and $S(z)=0$, so $z$ is insufficient to construct the function $g$. Equations \eqref{Gauss} are reduced to
\begin{align}
(X_1-iX_2)(a_6+i b_6)&=-i(b_1+ia_1)(a_6+i b_6),\label{gauab16}\\
X_3(a_6+i b_6)&=\pm i\frac{5}{3} \sqrt{\epsilon+a_3^2} (a_6+ib_6),\label{gauab36}\\
X_4(a_6+i b_6)&=2a_3(a_6+ib_6),\label{gauab46}\\
X_1(b_1)-X_2(a_1)&=2r^2-\frac{8}{3} (\epsilon+a_3^2)-a_1^2-b_1^2,\label{gauab121-}\\
X_1(a_1)+X_2(b_1)&=0,\label{gauab121+}\\
X_4(a_1+ib_1) &= a_3(a_1+ib_1)\pm\sqrt{\epsilon+a_3^2}\frac{(a_6+ib_6)}{3},\label{gauab41}\\
X_3(a_1+ib_1)&=\frac{i}{3}\left(a_3(a_6+ib_6)\pm 2\sqrt{\epsilon+a_3^2}(a_1+ib_1)\right).\label{gauab31}
\end{align}
The first equation is obtained from applying integrability on $a_3$. Now we define
$$w=\frac{a_6+i b_6}{\epsilon+a_3^2},$$
which after derivation gives
\begin{align*}
X_4(w)&=-2 a_3 \frac{(a_6+i b_6)}{\epsilon+a_3^2}+2 a_3 \frac{a_6+i b_6}{\epsilon+a_3^2}=0,\\
S(w)&=\pm i\frac{5}{3}w.
\end{align*}
The resulting differential equation for a function $G$ of $w$ will be
\begin{equation*}
S(G)=\pm G' i\frac{5}{3} w=-1 \Leftrightarrow G'=\pm i \frac{3}{5w}.
\end{equation*}
The solution is that $G$ is a logarithm of $w$. We find that $H$ defined by
\begin{equation*}
H=w^2(\epsilon+a_3^2)^2\rho^5 r
\end{equation*}
is a constant and hence can be used to express $\rho$. We can thus solve $w$ as
\begin{equation*}
w=e^{k_1\pm i\frac{5}{3}s}=e^{k_1}\left(\cos(\frac{5}{3}s)\pm i\sin(\frac{5}{3}s)\right).
\end{equation*}
Applying \eqref{gaua43},\eqref{gauab46},\eqref{gauab36},\eqref{cod4r} and \eqref{cod3r} when $\epsilon=0$ yields
\begin{align*}
a_3&=-\frac{1}{t},\\
a_6&=e^{k_1}\frac{\cos(\frac{5}{3}s)}{t^2},\\
b_6&=\pm e^{k_1} \frac{\sin(\frac{5}{3}s)}{t^2},\\
r&=\frac{e^{k_2}}{t}.
\end{align*}
The equation \eqref{cod1r} now gives $a_1+ib_1$ immediately without going through $\gamma_1$ because of \eqref{cod3r}. The final unknown, $\gamma_1$ can then be determined using \eqref{gauab16}. When we pick $b_2=a_3$, we obtain
\begin{align*}
\gamma_{1+}&=\frac{-5 e^{k_1+\frac{5si}{3}}+e^{\frac{2k_1+k_2}{5}+\frac{2si}{3}}t\left((\pard{k_2}{v}-3\pard{k_1}{v})-i(\pard{k_2}{u}-3\pard{k_1}{u})\right)}{5t^2},\\
a_{1+}&=\frac{e^{k_1}\cos(\frac{5s}{3})+e^{\frac{2k_1+k_2}{5}}t\left(\cos(\frac{2s}{3})\pard{k_2}{v}+\sin(\frac{2s}{3})\pard{k_2}{u}\right)}{3t^2},\\
b_{1+}&=\frac{e^{k_1}\sin(\frac{5s}{3})-e^{\frac{2k_1+k_2}{5}}t\left(\cos(\frac{2s}{3})\pard{k_2}{u}-\sin(\frac{2s}{3})\pard{k_2}{v}\right)}{3t^2},\\
\end{align*}
and for $b_2=-a_3$ we obtain
\begin{align*}
\gamma_{1-}&=\frac{-5 e^{k_1-\frac{5si}{3}}+e^{\frac{2k_1+k_2}{5}-\frac{2si}{3}}t\left((3\pard{k_1}{v}-\pard{k_2}{v})-i(3\pard{k_1}{u}-\pard{k_2}{u})\right)}{5t^2},\\
a_{1-}&=\frac{-e^{k_1}\cos(\frac{5s}{3})+e^{\frac{2k_1+k_2}{5}}t\left(\cos(\frac{2s}{3})\pard{k_2}{v}-\sin(\frac{2s}{3})\pard{k_2}{u}\right)}{3t^2},\\
b_{1-}&=\frac{e^{k_1}\sin(\frac{5s}{3})-e^{\frac{2k_1+k_2}{5}}t\left(\cos(\frac{2s}{3})\pard{k_2}{u}+\sin(\frac{2s}{3})\pard{k_2}{v}\right)}{3t^2}.\\
\end{align*}
 Equations \eqref{gauab121-} and \eqref{coor12g1}  result in restrictions on the functions $k_1$ and $k_2$ of $(u,v)$ given by
\begin{equation}
\begin{split}
\Delta k_2 &= e^{-\frac{2}{5}(2k_1+k_2)}(8-6e^{2k_2}),\\
\Delta k_1 &= e^{-\frac{2}{5}(2k_1+k_2)}(6-2e^{2k_2}). \label{constrab}
\end{split}
\end{equation}
These equations are valid for both $b_2=\pm a_3$. Using the constructed functions, the rest of the Gauss equations don't impose further conditions. We can summarize this result in the following theorem.
\begin{theorem}
Each special Lagrangian submanifold of $\Cx^4$ with $SO(2)\rtimes S_3$-symmetry where the only integral distribution containing $\mathcal{N}_1$ is the tangent bundle, can be constructed in the way above using either functions $\{k_1,k_2,k_3,k_4\}$ subject to \eqref{constraa} or functions $\{k_1,k_2\}$ subject to \eqref{constrab}. Conversely, each such a construction results in such a submanifold, unique up to local isometry.
\end{theorem}
In the upcoming sections we will consider the construction for $\epsilon=\pm 1$.

\section{Submanifolds in $\Cx P^4$.}
\subsection{The case where $b_2= 0$.} This means that both $\mathcal{N}_1$ and $\mathcal{N}_2$ are integrable distributions. We can assume $a_3=0$.  However, the Gauss equation
\begin{equation}
X_3(a_2)=1+a_2^2 \label{1+a2k}
\end{equation}
no longer allows for $a_2$ being a constant. The following result is obtained:
\begin{theorem}
Suppose $M$ is a special Lagrangian submanifold in $\Cx P^4$ having $SO(2)\rtimes S_3$-symmetry. Suppose $\mathcal{N}_1$ is integrable. Then $M$ can be lifted horizontally to a submanifold in the unit sphere of $\Cx^5$ through $F$ and this lift is congruent to
\begin{equation}
F(t,s,u,v)=\left(\phi(u,v) \cos(s),\sin(s)\cos(t),\sin(s)\sin(t)\right),\label{cp1}
\end{equation}
where $\phi$ is the horizontal lift of a special Lagrangian submanifold in $\Cx P^2$ to the unit sphere in $\Cx^3$.
\end{theorem}
\begin{proof}
Equations \eqref{Gauss} reduce $\nabla$ to
\begin{align*}
\begin{matrix}
&\nabla_{X_1}X_1=a_1X_2+a_2X_3&\vline & \nabla_{X_1}X_2=-a_1X_1,\\
&\nabla_{X_2}X_1=b_1X_2 &\vline & \nabla_{X_2}X_2=-b_1X_1+a_2X_3,\\
&\nabla_{X_3}X_1=0 &\vline & \nabla_{X_3}X_2=0,\\
&\nabla_{X_4}X_1=0 &\vline & \nabla_{X_4}X_2=0,\\
&\nabla_{X_1}X_3=-a_2X_1 &\vline & \nabla_{X_1}X_4=0,\\
&\nabla_{X_2}X_3=-a_2X_2&\vline & \nabla_{X_2}X_4=0,\\
&\nabla_{X_3}X_3=0 &\vline & \nabla_{X_3}X_4=0,\\
&\nabla_{X_4}X_3=\frac{X_4}{a_2} &\vline & \nabla_{X_4}X_4=-\frac{X_3}{a_2}.\\
\end{matrix}
\end{align*}
The distributions $\mathcal{N}_1$ and $\mathcal{N}_2$ satisfy the conditions for a warped product $N_2\times_{e^f}N_1$. Furthermore, the distributions $\spa\{X_3\}$ and $\spa\{X_4\}$ satisfy those of a warped product and we can write $M=\Rl\times_{e^g}\Rl\times_{e^f} N_1$. The functions $f$ and $g$ depend solely on the parameter corresponding to $X_3$ and are given by $X_3(f)=-a_2$ and $X_3(g)=\frac{1}{a_2}$. We can assume $X_3=\pard{}{s}$ on the submanifold. We can also find a function $\mu(s)$ such that $\mu X_4=\pard{}{t}$. To find a suitable $\mu$, we solve
\begin{equation*}
\lie{X_3,\mu X_4}=\left(X_3(\mu)-\frac{\mu}{a_2}\right)X_4=\left(\mu'(1+a_2^2)-\frac{\mu}{a_2}\right)X_4=0.
\end{equation*}
The function $\mu=\frac{a_2}{\sqrt{1+a_2^2}}$ satisfies this equation. We can find $a_2(s)$ by solving
\begin{equation*}
\pard{a_2}{s}=1+a_2^2\Rightarrow a_2=\tan(s).
\end{equation*}
Hence $\mu(s)=\sin(s)$ and we calculate for $i\in \{1,2\}$ that
\begin{align*}
D_{\pard{}{s}}\pard{}{s}&=\pardd{F}{s}=-F\\
          &\Rightarrow F=A \cos(s)+B\sin(s),\\
D_{\pard{}{t}}\pard{}{s}&=\pards{F}{t}{s}=-\pard{A}{t}\sin(s)+\pard{B}{t}\cos(s)\\
      &=\cot(s)\pard{F}{t}=\frac{\cos(s)^2}{\sin(s)}\pard{A}{t}+\cos(s)\pard{B}{t}\\
      &\Rightarrow \pard{A}{t}=0,\\
D_{X_i}\pard{}{s}&=\pard{F_*X_i}{s}=-A_*X_i\sin(s)+B_*X_i\cos(s)\\
          &=-\tan(s) X_i=-A_*X_i \sin(s)-\frac{\sin(s)^2}{\cos(s)} B_*X_i\\
          &\Rightarrow B_*X_i=0.
\end{align*}
So $A$ is the immersion of $N_1$ and $B$ is a curve tangent to $X_4$. Because $F$ lies in the unit sphere, one has
\begin{equation*}
\scal{F,F}=\cos(s)^2\scal{A,A}+\sin(s)^2\scal{B,B}+\sin(2s)\scal{A,B}
\end{equation*}
which implies that $A$ and $B$ have both unit length and are orthogonal. We can also calculate
\begin{align*}
D_{\pard{}{t}}\pard{}{t}&=-\cos(s)\sin(s)\pard{F}{s}-\sin(s)^2 F=-\sin(s)B\\
    &=\pardd{F}{t}=\sin(s)\pardd{B}{t}\\
    &\Rightarrow B=B_1\cos(t)+B_2\sin(t).
\end{align*}
Vector fields $B_1$ and $B_2$ are constant, normalized and orthogonal. This follows from the fact that $\scal{B,B}=1$. Finally similar to \eqref{refren}, $A$ can be shown to be any special Lagrangian submanifold in $\Cx P^2$ lifted to the unit sphere in $\Cx^3$ orthogonal to $B_1$ and $B_2$ directions. Fixing $B_1$ and $B_2$ by an isometry leads to \eqref{cp1}.
\end{proof}
\subsection{The case where $b_2\neq0$.}
When $\mathcal{N}_+$ is integrable, so when $a_6=b_6=0$, the equations for $\nabla$ are given by \eqref{Conmat2}. We have:
\begin{theorem}
Suppose $M$ is a special Lagrangian submanifold in $\Cx P^4$ having a local $SO(2)\rtimes S_3$ symmetry. Suppose $\mathcal{N}_+$ is integrable. Then $M$ can be lifted horizontally to a submanifold in the unit sphere of $\Cx^5$ through $F$ and is locally isometric to
\begin{equation}
F(t,s,u,v)=\left(\phi(s,u,v) \cos(t),\sin(t)\right),\label{cp2}
\end{equation}
where $\phi$ is the horizontal lift of a special Lagrangian submanifold in $\Cx P^3$ with $S_3$-symmetry to the unit sphere in $\Cx^4$.
\end{theorem}
\begin{proof}
The manifold is a warped product $\Rl\times_{e^f} N^3$. Solving the Gauss equation
\begin{equation*}
X_4(a_3)=\frac{\partial a_3}{\partial t}=1+a_3^2
\end{equation*}
yields $a_3=\tan(t)$. For $i\in \{1,2,3\}$ this implies
\begin{align*}
D_{X_4}X_4&=\pardd{F}{t}=-F\\
          &\Rightarrow F=A\cos(t)+B\sin(t),\\
D_{X_i}X_4&=-A_*X_i \sin(t)+B_*X_i\cos(t)\\
          &=-\tan(t) X_i=-A_*X_i\sin(t)-B_*X_i\frac{\sin(t)^2}{\cos(t)}\\
          &\Rightarrow B_*=0.
\end{align*}
Thus $B$ is a constant vector field along the submanifold and $A$ is an immersion of a $3$-fold $N^3$. Using the fact that $F$ is of unit length, $A$ and $B$ are orthogonal and of unit length. Using calculations similar to \eqref{refren} $A$ is a C-totally real submanifold in $S^7$ having local $S_3$-symmetry, where $S^7$ lies in the subspace orthogonal to $B$ and $JB$. Applying a suitable isometry results in \eqref{cp2}.
\end{proof}
The method to solve the case where the only integrable distribution containing $\mathcal{N}_1$ is the tangent bundle, has been analyzed earlier for a non-specific complex space form. We can now fill in $\epsilon=1$ and we find for $z^2\neq -1$ that
\begin{align*}
a_3 &= \tan(t),\\
a_2&=\frac{\sin(2 s)}{\cos(t)(\cos(2 s)+\cosh(2 k_1))},\\
b_2&=\frac{\sinh(2 k_1)}{\cos(t)(\cos(2 s)+\cosh(2 k_1))},\\
r&=\frac{e^{k_2}}{\cos(t) \sqrt{\cos(2s)+\cosh(2 k_1)}},\\
a_6+ib_6&=\frac{k_3+ik_4}{\rho} \sqrt{1+a_3^2} \left(1+\bar{z}^2\right)^{-\frac{1}{2}},
\end{align*}
where the functions $k_i$ depend only on $(u,v)$. Solving for \eqref{Gauss}, we obtain furthermore that
\begin{align*}
\gamma_1&=\frac{-\tan(t)(k_3+ik_4)\cos(s-ik_1)+\left(\pard{k_1}{v}-i\pard{k_1}{u}\right)}{\rho},\\
a_1&=\frac{2^{\frac{2}{3}}e^{\frac{2}{3}k_2}}{3\cos(t)^2\left(\cos(2s)+\cosh(2k_1)\right)^2}\big((\cos(2s)+\cosh(2k_1))(\rho_1\pard{k_2}{v}+\rho_2\pard{k_2}{u})\\
&\qquad+\sin(2s)(\rho_1\pard{k_1}{u}-\rho_2\pard{k_1}{v})-\sinh(2k_1)(\rho_1\pard{k_1}{v}+\rho_2\pard{k_1}{u})\\
&\qquad -\tan(t) \sinh(2k_1) \left(\cos(s)\cosh(k_1)(k_4\rho_2-k_3\rho_1)+\sin(s)\sinh(k_1)(k_4\rho_1+k_3\rho_2)\right)\big),\\
b_1&=\frac{2^{\frac{2}{3}}e^{\frac{2}{3}k_2}}{3\cos(t)^2\left(\cos(2s)+\cosh(2k_1)\right)^2}\big((\cos(2s)+\cosh(2k_1))(\rho_2\pard{k_2}{v}-\rho_1\pard{k_2}{u})\\
&\qquad+\sin(2s)(\rho_2\pard{k_1}{u}+\rho_1\pard{k_1}{v})+\sinh(2k_1)(\rho_1\pard{k_1}{u}-\rho_2\pard{k_1}{v})\\
&\qquad - \sinh(2k_1)\tan(t) \left(\sin(s)\sinh(k_1)(k_4\rho_2-k_3\rho_1)-\cos(s)\cosh(k_1)(k_4\rho_1+k_3\rho_2)\right)\big).
\end{align*}
The other equations in \eqref{Gauss} impose restrictions on $\{k_1,k_2,k_3,k_4\}$ given by
\begin{equation}
\begin{split}
\pard{k_4}{u}-\pard{k_3}{v}&=2\tanh (k_1) \left(k_3 \pard{k_1}{v}-k_4 \pard{k_1}{u}\right),\\
\pard{k_4}{v}+\pard{k_3}{u}&=-2\coth(k_1) \left(k_3 \pard{k_1}{u}+k_4 \pard{k_1}{v}\right),\\
\Delta k_1&=e^{-\frac{2k_2}{3}} \frac{\sinh(2k_1)}{2} \left(-2^{\frac{4}{3}}+e^{\frac{2k_2}{3}}(k_3^2+k_4^2)\right),\\
\Delta k_2&=3*2^{\frac{1}{3}} e^{-\frac{2 k_2}{3}}\left(\cosh(2k_1)- e^{2k_2}\right).\label{cpk1}
\end{split}
\end{equation}
When $a_2=0$ and $b_2=\pm\sqrt{1+a_3^2}$, we find
\begin{align*}
a_6&=\frac{e^{k_1} \cos(\frac{5}{3} s)}{\cos(t)^2};\\
b_6&=\pm\frac{e^{k_1} \sin(\frac{5}{3} s)}{\cos(t)^2};\\
r&=\frac{e^{k_2}}{\cos(t)}.
\end{align*}
Furthermore, we obtain for $b_2=\sqrt{1+a_3^2}$ that
\begin{align*}
\gamma_{1+}&=\frac{-5 e^{k_1+\frac{5si}{3}}\tan(t)+e^{\frac{2k_1+k_2}{5}+\frac{2si}{3}}\left((\pard{k_2}{v}-3\pard{k_1}{v})-i(\pard{k_2}{u}-3\pard{k_1}{u})\right)}{5\cos(t)},\\
a_{1+}&=\frac{e^{k_1}\cos(\frac{5s}{3})\tan(t)+e^{\frac{2k_1+k_2}{5}}\left(\cos(\frac{2s}{3})\pard{k_2}{v}+\sin(\frac{2s}{3})\pard{k_2}{u}\right)}{3\cos(t)},\\
b_{1+}&=\frac{e^{k_1}\sin(\frac{5s}{3})\tan(t)-e^{\frac{2k_1+k_2}{5}}\left(\cos(\frac{2s}{3})\pard{k_2}{u}-\sin(\frac{2s}{3})\pard{k_2}{v}\right)}{3\cos(t)},\\
\end{align*}
and for $b_2=-\sqrt{1+a_3^2}$ we obtain
\begin{align*}
\gamma_{1-}&=\frac{-5 e^{k_1-\frac{5si}{3}}\tan(t)+e^{\frac{2k_1+k_2}{5}-\frac{2si}{3}}\left((3\pard{k_1}{v}-\pard{k_2}{v})-i(3\pard{k_1}{u}-\pard{k_2}{u})\right)}{5\cos(t)},\\
a_{1-}&=\frac{-e^{k_1}\cos(\frac{5s}{3})\tan(t)+e^{\frac{2k_1+k_2}{5}}\left(\cos(\frac{2s}{3})\pard{k_2}{v}-\sin(\frac{2s}{3})\pard{k_2}{u}\right)}{3\cos(t)},\\
b_{1-}&=\frac{e^{k_1}\sin(\frac{5s}{3})\tan(t)-e^{\frac{2k_1+k_2}{5}}\left(\cos(\frac{2s}{3})\pard{k_2}{u}+\sin(\frac{2s}{3})\pard{k_2}{v}\right)}{3\cos(t)}.\\
\end{align*}
Solving the last equations in \eqref{Gauss} implies restrictions on the functions $k_1(u,v)$ and $k_2(u,v)$ given by
\begin{equation}
\begin{split}
\Delta k_1&=2e^{-\frac{2(2k_1+k_2)}{5}}\left(3-e^{2k_1}-e^{2 k_2}\right)\\
\Delta k_2&=e^{-\frac{2(2k_1+k_2)}{5}}\left(8-e^{2k_1}-6 e^{2 k_2}\right).\label{cpk2}
\end{split}
\end{equation}
These equations are valid for both $b_2=\pm\sqrt{1+a_3^2}$. We summarize this in the following theorem.
\begin{theorem}
Each special Lagrangian submanifold of $\Cx P^4$ with $SO(2)\rtimes S_3$-symmetry where the only integral distribution containing $\mathcal{N}_1$ is the tangent bundle, can be constructed in the way above using either functions $\{k_1,k_2,k_3,k_4\}$ subject to \eqref{cpk1} or functions $\{k_1,k_2\}$ subject to \eqref{cpk2}. Conversely, each such a construction results in such a submanifold, unique up to local isometry.
\end{theorem}

\section{Submanifolds in $\Cx H^4$.}
\subsection{The case where $b_2=0$.}
This is the case where $\mathcal{N}_1$ is an integrable distribution. We assume that $a_3=0$. Similar to the case in $\Cx P^4$ we have that $M$ is a double warped product $\Rl\times_{e^g}\Rl\times_{e^f} N^2$. the function $a_2$ depends only on the coordinate $s$ and is given by
\begin{equation*}
\pard{a_2}{s}=a_2^2-1.
\end{equation*}
This equation has $3$ possible solutions, depending on the initial conditions. For $a_2(0)=1$, it is a constant. For $a_2(0)>1$ it is given as $a_2=-\coth(s)$. Finally for $a_2(0)<1$, it is given as $a_2(s)=-\tanh(s)$. The connection $\nabla$ is given by
\begin{align*}
\begin{matrix}
&\nabla_{X_1}X_1=a_1X_2+a_2X_3&\vline & \nabla_{X_1}X_2=-a_1X_1,\\
&\nabla_{X_2}X_1=b_1X_2 &\vline & \nabla_{X_2}X_2=-b_1X_1+a_2X_3,\\
&\nabla_{X_3}X_1=0 &\vline & \nabla_{X_3}X_2=0,\\
&\nabla_{X_4}X_1=0 &\vline & \nabla_{X_4}X_2=0,\\
&\nabla_{X_1}X_3=-a_2X_1 &\vline & \nabla_{X_1}X_4=0,\\
&\nabla_{X_2}X_3=-a_2X_2&\vline & \nabla_{X_2}X_4=0,\\
&\nabla_{X_3}X_3=0 &\vline & \nabla_{X_3}X_4=0,\\
&\nabla_{X_4}X_3=-\frac{X_4}{a_2} &\vline & \nabla_{X_4}X_4=\frac{X_3}{a_2}.\\
\end{matrix}
\end{align*}
We have the following result.
\begin{theorem}
Suppose $M$ is a special Lagrangian submanifold in $\Cx H^4$ having a local $SO(2)\rtimes S_3$-symmetry. Suppose $\mathcal{N}_1$ is integrable. Then $M$ can be lifted horizontally to a submanifold in $H^9$ through $F$ and is locally isometric to either
\begin{equation}
F(t,s,u,v)=\left(\sin(t)\sinh(s),\cos(t)\sinh(s),\phi(u,v)\cosh(s)\right),\label{ch1}
\end{equation}
where $\phi$ is the horizontal lift of a special Lagrangian submanifold of $\Cx H^2$ to $H^5$ in case $a_2^2<1$, or
\begin{equation}
F(t,s,u,v)=\left(\phi(u,v)\sinh(s),\cos(t)\cosh(s),\sin(t)\cosh(s)\right),\label{ch12}
\end{equation}
where $\phi$ is the horizontal lift of a special Lagrangian submanifold of $\Cx P^2$ to $S^5$ in case $a_2^2>1$, or
\begin{equation}
F(t,s,u,v)=\Big(\left(\phi(u,v),t\right)e^{-s},-\frac{1}{2} e^{-s},\left(\norm{\left(\phi(u,v),t\right)}^2+if(u,v)\right)e^{-s}+e^s\Big),\label{ch2}
\end{equation}
where $\phi$ is a special Lagrangian surface in $\Cx^2$ and $f$ is the integral of the differential form
\begin{equation*}
2\sum_{i=1}^2 \left(x^i \dv y^i-y^i \dv x^i\right)
\end{equation*}
on $\Cx^2$ in case $a_2^2=1$.
\end{theorem}
\begin{proof}
We can check similar to the case in $\Cx P^4$ that $M=\Rl\times_{e^g}\Rl\times_{e^f}\times N^2$, where $f$ and $g$ are functions on the first factor, determined by $X_3(g)=-\frac{1}{a_2}$ and $X_3(f)=-a_2$. We can treat the cases separately for each solution to $a_2(s)$.

Assume $a_2=-\tanh(s)$, then it is easy to see that $\pard{}{t}=-\sinh(s) X_4$ commutes with $\pard{}{s}$. The Gauss identity now implies for $i \in \{1,2\}$ that
\begin{align*}
D_{\pard{}{s}}\pard{}{s}&=\pardd{F}{s}=F\\
          &\Rightarrow F=A\sinh(s)+B\cosh(s),\\
D_{\pard{}{T}}\pard{}{s}&=\pards{F}{t}{s}=\pard{A}{t}\cosh(s)+\pard{B}{t}\sinh(s)\\
      &=\coth(s)\pard{F}{t}=\pard{A}{t}\cosh(s)+\pard{B}{t}\frac{\cosh(s)^2}{\sinh(s)}\\
      &\Rightarrow \pard{B}{t}=0,\\
D_{X_i}\pard{}{s}&=\pard{F_*X_i}{s}=A_*X_i\cosh(s)+B_*X_i\sinh(s)\\
          &=\tanh(s) X_i=A_*X_i \frac{\sinh(s)^2}{\cosh(s)}+B_*X_i\sinh(s)\\
          &\Rightarrow A_*X_i=0.
\end{align*}
Using the fact that $\scal{F,F}_1=-1$, we get that $\scal{B,B}_1=-\scal{A,A}_1=-1$ and $\scal{A,B}_1=0$. Furthermore, we find
\begin{equation*}
\begin{split}
D_{\pard{}{t}}\pard{}{t}&=\pardd{F}{t}=\pardd{A}{t}\sinh(s)\\
    &=-\cosh(s)\sinh(s)\pard{F}{s}+\sinh(s)^2F=-A \sinh(s)\\
    &\Rightarrow A=A_1\cos(t)+A_2\sin(t).
    \end{split}
\end{equation*}
Because $A$ has unit length, so do $A_1$ and $A_2$ and they are both orthogonal. Calculations similar to \eqref{refren} show that $B$ can be taken as the horizontal lift of any special Lagrangian submanifold in $\Cx H^2$ and applying a suitable isometry gives \eqref{ch1}.\\
For $a_2=-\coth(s)$ calculations similar to the previous case result in \eqref{ch12}.\\
Finally we assume $a_2=1$. Then the vector field given by $\pard{}{t}=e^{-s} X_4$ commutes with $\pard{}{s}$. We can calculate for $i\in \{1,2\}$ that
\begin{align*}
D_{\pard{}{s}}\pard{}{s}&=\pardd{F}{s}=F \Rightarrow F=A e^{s}+B e^{-s},\\
D_{\pard{}{t}}\pard{}{s}&=\pards{F}{t}{s}=\pard{A}{t} e^s-\pard{B}{t} e^{-s}\\
      &=-\pard{F}{t}=-\pard{A}{t} e^s-\pard{B}{t} e^{-s}\\
      &\Rightarrow \pard{A}{t}=0,\\
D_{X_i}\pard{}{s}&=\pard{F_*X_i}{s}=A_*X_i e^s-B_*X_i e^{-s}\\
          &=-F_*X_i=-\left(A_*X_i e^s+B_*X_i e^{-s}\right)\\
          &\Rightarrow A_*=0.
\end{align*}
Using the fact that $\scal{F,F}_1=-1$, we obtain that $A$ and $B$ are vector fields with $0$ length and they satisfy $\scal{A,B}_1=-\frac{1}{2}$. Further calculations show
\begin{align*}
D_{\pard{}{t}}\pard{}{t}&=\pardd{F}{t}=e^{-s} \pardd{B}{t}\\
    &= e^{-2s}\left(\pard{F}{s}+F\right)=2A e^{-s}\\
    &\Rightarrow B=A t^2+B_1t+B_2,\\
D_{X_i} \pard{}{t}&=\pard{F_*X_i}{t}=B_{1*}X_i e^{-s}=0\\
         &\Rightarrow B_{1*}=0.
\end{align*}
We can conclude that $F$ has the form
\begin{equation*}
F=\left(At^2+B_1t+\phi\right)e^{-s}+A e^s
\end{equation*}
Here, $\phi$ is an immersion of a $2$-fold in $\Cx^5_1$ tangent to $\mathcal{N}_1$. Calculating the scalar products of $B$ and $A$, we get
\begin{equation}
\begin{split}
\scal{A,B}_1&=t\scal{A,B_1}_1+\scal{A,\phi}_1=-\frac{1}{2}\\
            &\Rightarrow \scal{A,B_1}_1=0\bigwedge\scal{A,\phi}_1=-\frac{1}{2},\\
\scal{B,B}_1&=t^2\left(\scal{B_1,B_1}_1-1\right)+2t\scal{B_1,\phi}_1+\scal{\phi,\phi}_1=0\\
            &\Rightarrow \scal{B_1,B_1}_1=1\bigwedge\scal{B_1,\phi}_1=0\bigwedge\scal{\phi,\phi}_1=0.\label{scala}
\end{split}
\end{equation}
We can shift to a different standard basis of $\Cx^5_1$ such that
\begin{equation*}
\scal{\vec{z},\vec{w}}=\Re\left(\sum_{j=1}^3 z_j \bar{w}_j+ z_4\bar{w}_5+z_5\bar{w}_4\right).
\end{equation*}
In this case the constant light-like vector $A$ and time-like $B_1$, after applying a suitable isometry, can be chosen to be
\begin{align*}
A&=\left(0,0,0,0,1\right),\\
B_1&=\left(0,0,1,0,0\right).
\end{align*}
We can write $\phi=(\phi_1,\phi_2,\phi_3,\phi_4,\phi_5)$ where $\phi_j=x_j+iy_j$. Then \eqref{scala} implies
\begin{align*}
x_4&=-\frac{1}{2},\\
x_3&=0,\\
x_5-2y_4 y_5&=\sum_{j=1}^3\abs{\phi_j}^2.
\end{align*}
We can use the fact that both $F$ and $iF$ are orthogonal to the tangent space in $\Cx^5_1$ and this results in
\begin{align*}
\phi_4&=-\frac{1}{2},\\
\phi_3&=0,\\
\dv y_5&=2\sum_{i=1}^2(x_i \dv y_i-y_i \dv x_i).
\end{align*}
This last equation is integrable if and only its derivative equals $0$ along the submanifold. But this derivative is nothing more than a multiple of the K\"ahler form on $\Cx^2$ spanned by the first $2$ complex coordinates. In other words, for such a submanifold to exist, the projection of $\phi$ onto the first 2 coordinates should be a Lagrangian submanifold in $\Cx^2$. Calculating the Gauss identity on $D_{X_i}X_j$ we find that the metric on this immersion is given by
\begin{equation*}
\scal{\phi_*X_i,\phi_*X_j}=e^{2s} \delta_{ij},
\end{equation*}
where $\scal{a,b}$ is the standard scalar product on $\Cx^2$ and $\phi$ here is the restriction to the first $2$ complex coordinates. Because $\scal{F_*X_i,F_*X_j}=\delta_{ij}$ and because $\phi_{3*}=0$ and $\phi_{4*}=0$ this condition is included in the warped product structure. Using calculations like \eqref{refren} we conclude that $(\phi_1,\phi_2)$ can be any special Lagrangian $2$-fold in $\Cx^2$. The result is summarized in \eqref{ch2}.  \end{proof}

\subsection{The case where $b_2 \neq 0$.}
First we assume that $\mathcal{N}_+$ is an integrable distribution. This is equivalent to $a_6+ib_6=0$. The connection is given by \eqref{Conmat2}, resulting in a warped product structure $\Rl\times_{e^f}N^3 $. The equation
\begin{equation*}
X_4(a_3)=\pard{a_3}{t}=a_3^2-1
\end{equation*}
has a solution given as $\abs{a_3}=1$, $a_3=-\tanh(t)$ or $a_3=-\coth(t)$, depending on the initial value of $a_3$. Using an analysis similar to the case of $\Cx P^4$ and the case above gives the following result.
\begin{theorem}
Suppose $M$ is a special Lagrangian submanifold in $\Cx H^4$ having a local $SO(2)\rtimes S_3$-symmetry. Suppose $\mathcal{N}_1$ is non-integrable, but is contained in the integrable $\mathcal{N}_+$ distribution. Then $M$ can be lifted horizontally to a submanifold in $H^9$ through $F$ and is locally isometric to either
\begin{equation}
F(t,s,u,v)=\left(\sinh(t),\phi(s,u,v)\cosh(t)\right),\label{ch3}
\end{equation}
where $\phi$ is the horizontal lift of a special Lagrangian submanifold of $\Cx H^3$ with a local $S_3$-symmetry to $H^7$ in case $a_3^2<1$, or
\begin{equation}
F(t,s,u,v)=\left(\phi(s,u,v)\sinh(t),\cosh(t)\right),\label{ch32}
\end{equation}
where $\phi$ is the horizontal lift of a special Lagrangian submanifold of $\Cx P^3$ with a local $S_3$-symmetry to $S^7$ in case $a_3^2>1$, or
\begin{equation}
F(t,s,u,v)=\big(\phi(s,u,v)e^{-t},-e^{-t}/2,\left(\norm{\phi(s,u,v)}^2+if(s,u,v)\right)e^{-t}+e^t\big),\label{ch4}
\end{equation}
where $\phi$ is a special Lagrangian submanifold with a local $S_3$-symmetry  in $\Cx^3$ and $f$ is the integral of the differential form
\begin{equation*}
2\sum_{i=1}^3 \left(x^i \dv y^i-y^i \dv x^i\right).
\end{equation*}
\end{theorem}
Finally we assume that there is no integrable distribution that contains $\mathcal{N}_1$ except for the tangent bundle. We return to the analysis as done for $\Cx^4$, but set $\epsilon=-1$. The result will depend on the initial value of $a_3$. First assume that $a_3^2<1$, then $\tilde{\epsilon}=-1$. We find functions $\{k_1,k_2,k_3,k_4\}$ of $(u,v)$ such that
\begin{align*}
a_3&=-\tanh(t),\\
a_2&=-\frac{\sinh(2s)}{\cosh(t) \left(\cosh(2s)+\cos(2 k_1)\right)},\\
b_2&=-\frac{\sin(2k_1)}{\cosh(t) \left(\cosh(2s)+\cos(2 k_1)\right)},\\
r&=\frac{e^{k_2}}{\cosh(t)\sqrt{\cosh(2s)+\cos(2k_1)}},\\
a_6+ib_6&=\frac{k_3+ik_4}{\rho} \sqrt{1-a_3^2}\left(1-\bar{z}^2\right)^{-\frac{1}{2}}.\\
\end{align*}
Using \eqref{Gauss} as earlier, we obtain $a_1$, $b_1$, $\gamma_1$ as
\begin{align*}
\gamma_1&=\frac{-\tanh(t)(k_3+ik_4)\cosh(s-ik_1)+\left(\pard{k_1}{v}-i\pard{k_1}{u}\right)}{\rho},\\
a_1&=\frac{2^{\frac{2}{3}}e^{\frac{2}{3}k_2}}{3\cosh(t)^2\left(\cosh(2s)+\cos(2k_1)\right)^2}\big((\cosh(2s)+\cos(2k_1))(\rho_1\pard{k_2}{v}+\rho_2\pard{k_2}{u})\\
&\qquad+\sinh(2s)(\rho_2\pard{k_1}{v}-\rho_1\pard{k_1}{u})+\sin(2k_1)(\rho_1\pard{k_1}{v}+\rho_2\pard{k_1}{u})\\
&\qquad + \sin(2k_1)\tanh(t) \left(\cosh(s)\cos(k_1)(k_4\rho_2-k_3\rho_1)-\sinh(s)\sin(k_1)(k_4\rho_1+k_3\rho_2)\right)\big),\\
b_1&=\frac{2^{\frac{2}{3}}e^{\frac{2}{3}k_2}}{3\cosh(t)^2\left(\cosh(2s)+\cos(2k_1)\right)^2}\big((\cosh(2s)+\cos(2k_1))(\rho_2\pard{k_2}{v}-\rho_1\pard{k_2}{u})\\
&\qquad-\sinh(2s)(\rho_2\pard{k_1}{u}+\rho_1\pard{k_1}{v})-\sin(2k_1)(\rho_1\pard{k_1}{u}-\rho_2\pard{k_1}{v})\big)\\
&\qquad - \sin(2k_1)\tanh(t) \left(\sinh(s)\sin(k_1)(k_4\rho_2-k_3\rho_1)+\cosh(s)\cos(k_1)(k_4\rho_1+k_3\rho_2)\right)\big).
\end{align*}
The other equations in \eqref{Gauss} put restrictions on $\{k_1,k_2,k_3,k_4\}$ given by
\begin{equation}
\begin{split}
\pard{k_4}{u}-\pard{k_3}{v}&=2\tan(k_1)\left(k_4\pard{k_1}{u}-k_3\pard{k_1}{v}\right),\\
\pard{k_4}{v}+\pard{k_3}{u}&=-2\cot(k_1)\left(k_3\pard{k_1}{u}+k_4\pard{k_1}{v}\right),\\
\Delta k_1&=\frac{\sin(2 k_1)}{2} \left(2^{\frac{4}{3}}e^{-\frac{2}{3} k_2}+k_3^2+k_4^2\right),\\
\Delta k_2&=-3*2^{\frac{1}{3}} e^{-\frac{2}{3} k_2}\left(e^{2 k_2}+\cos(2k_1)\right).\label{kh}
\end{split}
\end{equation}
Then we set $a_3^2>1$ and assume $z^2\neq -1$. We then find
\begin{align*}
a_3&=-\coth(t),\\
a_2&=\frac{\sin(2s)}{\sinh(t) \left(\cos(2s)+\cosh(2 k_1)\right)},\\
b_2&=\frac{\sinh(2k_1)}{\sinh(t) \left(\cos(2s)+\cosh(2 k_1)\right)},\\
r&=\frac{e^{k_2}}{\sinh(t)\sqrt{\cos(2s)+\cosh(2k_1)}},\\
a_6+ib_6&=\frac{k_3+ik_4}{\rho} \sqrt{a_3^2-1}\left(1+\bar{z}^2\right)^{-\frac{1}{2}}.\\
\end{align*}
We obtain
\begin{align*}
\gamma_1&=\frac{\coth(t)(k_3+ik_4)\cos(s-ik_1)+\left(\pard{k_1}{v}-i\pard{k_1}{u}\right)}{\rho},\\
a_1&=\frac{2^{\frac{2}{3}}e^{\frac{2}{3}k_2}}{3\sinh(t)^2\left(\cos(2s)+\cosh(2k_1)\right)^2}\big((\cos(2s)+\cosh(2k_1))(\rho_1\pard{k_2}{v}+\rho_2\pard{k_2}{u})\\
&\qquad+\sin(2s)(\rho_1\pard{k_1}{u}-\rho_2\pard{k_1}{v})-\sinh(2k_1)(\rho_1\pard{k_1}{v}+\rho_2\pard{k_1}{u})\\
&\qquad + \sinh(2k_1)\coth(t) \left(\cos(s)\cosh(k_1)(k_4\rho_2-k_3\rho_1)+\sin(s)\sinh(k_1)(k_4\rho_1+k_3\rho_2)\right)\big),\\
b_1&=\frac{2^{\frac{2}{3}}e^{\frac{2}{3}k_2}}{3\sinh(t)^2\left(\cosh(2s)+\cos(2k_1)\right)^2}\big((\cosh(2s)+\cos(2k_1))(\rho_2\pard{k_2}{v}-\rho_1\pard{k_2}{u})\\
&\qquad+\sin(2s)(\rho_2\pard{k_1}{u}+\rho_1\pard{k_1}{v})+\sinh(2k_1)(\rho_1\pard{k_1}{u}-\rho_2\pard{k_1}{v})\big)\\
&\qquad - \sinh(2k_1)\coth(t) \left(\sin(s)\sinh(k_1)(k_3\rho_1-k_4\rho_2)+\cos(s)\cosh(k_1)(k_4\rho_1+k_3\rho_2)\right)\big).
\end{align*}
The functions $\{k_1,k_2,k_3,k_4\}$ have to satisfy
\begin{equation}
\begin{split}
\pard{k_4}{u}-\pard{k_3}{v}&=2\tanh(k_1)\left(k_3\pard{k_1}{v}-k_4\pard{k_1}{u}\right),\\
\pard{k_4}{v}+\pard{k_3}{u}&=-2\coth(k_1)\left(k_3\pard{k_1}{u}+k_4\pard{k_1}{v}\right),\\
\Delta k_1&=-\sinh(2 k_1)\left(2^{\frac{1}{3}}e^{-\frac{2}{3} k_2}+\frac{k_3^2+k_4^2}{2}\right),\\
\Delta k_2&=3*2^{\frac{1}{3}} e^{-\frac{2}{3} k_2}\left(\cosh(2k_1)-e^{2 k_2}\right).\label{kh2}
\end{split}
\end{equation}
Finally, assume $a_3^2>1$ and $z^2=-1$. Then we find
\begin{align*}
a_3&=-\coth(t),\\
a_6&=\frac{e^{k_1}\cos(\frac{5}{3} s)}{\sinh(t)^2},\\
b_6&=\pm\frac{e^{k_1}\sin(\frac{5}{3} s)}{\sinh(t)^2},\\
r&=\frac{e^{k_2}}{\sinh(t)}.
\end{align*}
We obtain for $b_2=\sqrt{a_3^2-1}$ that
\begin{align*}
\gamma_{1+}&=\frac{5 e^{k_1+\frac{5si}{3}}\coth(t)+e^{\frac{2k_1+k_2}{5}+\frac{2si}{3}}\left((\pard{k_2}{v}-3\pard{k_1}{v})-i(\pard{k_2}{u}-3\pard{k_1}{u})\right)}{5\sinh(t)},\\
a_{1+}&=\frac{-e^{k_1}\cos(\frac{5s}{3})\coth(t)+e^{\frac{2k_1+k_2}{5}}\left(\cos(\frac{2s}{3})\pard{k_2}{v}+\sin(\frac{2s}{3})\pard{k_2}{u}\right)}{3\sinh(t)},\\
b_{1+}&=\frac{-e^{k_1}\sin(\frac{5s}{3})\coth(t)-e^{\frac{2k_1+k_2}{5}}\left(\cos(\frac{2s}{3})\pard{k_2}{u}-\sin(\frac{2s}{3})\pard{k_2}{v}\right)}{3\sinh(t)},\\
\end{align*}
and for $b_2=-\sqrt{a_3^2-1}$ we obtain
\begin{align*}
\gamma_{1-}&=\frac{5 e^{k_1-\frac{5si}{3}}\coth(t)+e^{\frac{2k_1+k_2}{5}-\frac{2si}{3}}\left((3\pard{k_1}{v}-\pard{k_2}{v})-i(3\pard{k_1}{u}-\pard{k_2}{u})\right)}{5\sinh(t)},\\
a_{1-}&=\frac{e^{k_1}\cos(\frac{5s}{3})\coth(t)+e^{\frac{2k_1+k_2}{5}}\left(\cos(\frac{2s}{3})\pard{k_2}{v}-\sin(\frac{2s}{3})\pard{k_2}{u}\right)}{3\sinh(t)},\\
b_{1-}&=\frac{-e^{k_1}\sin(\frac{5s}{3})\coth(t)-e^{\frac{2k_1+k_2}{5}}\left(\cos(\frac{2s}{3})\pard{k_2}{u}+\sin(\frac{2s}{3})\pard{k_2}{v}\right)}{3\sinh(t)}.\\
\end{align*}
 Solving the other Gauss equations results in the relations
\begin{equation}
\begin{split}
\Delta k_1&=e^{-\frac{2}{5}(2k_1+k_2)}\left(6+2e^{2k_1}-2e^{2k_2}\right),\\
\Delta k_2&=e^{-\frac{2}{5}(2k_1+k_2)}\left(8 + e^{2k_1} -6e^{2k_2}\right).\label{kh3}
\end{split}
\end{equation}
These equations are valid for both $b_2=\pm\sqrt{a_3^2-1}$. We can conclude with the following proposition.
\begin{theorem}
Each special Lagrangian submanifold of $\Cx H^4$ with $SO(2)\rtimes S_3$-symmetry where the only integral distribution containing  $\mathcal{N}_1$ is the whole tangent bundle, can be constructed in the way above using functions $\{k_1,k_2,k_3,k_4\}$ subject to \eqref{kh} in case $a_3^2(0)<1$, subject to \eqref{kh2} in case $a_3^2(0)>1$, or functions $\{k_1,k_2\}$ subject to \eqref{kh3} when $a_3^2(0)>1$ and $z^2=-1$. Conversely, each such a construction results in such a submanifold, unique up to local isometry.
\end{theorem}

\vskip 0.5cm


\begin{thebibliography}{112}

\bibitem{audin} M. Audin, Lagrangian submanifolds, http://www-irma.u-strasbg.fr/$\sim$maudin/newlagspe.ps,

\bibitem{bryant}  R. L. Bryant, \textit{Second order families of special Lagrangian 3-folds}, Perspectives in Riemannian geometry,  63–98,
CRM Proc. Lecture Notes, 40, Amer. Math. Soc., Providence, RI, 2006

\bibitem{Chen} B.-Y. Chen, \textit{Interactions of Legendre curves and Lagrangian submanifolds}, Israel J. of Mathematics 99 (1997), 69--108,

\bibitem{chenbook} B.-Y. Chen, Pseudo-Riemannian submanifolds, $\delta$-invariants and applications, World Scientific Publishing Company
(2011),


\bibitem{CDVV} B.-Y. Chen, F. Dillen, L. Verstraelen and L. Vrancken,
\textit{An exotic totally real minimal immersion of $S^3$ in $\Cx P^3$
and its characterization}, Proc. Royal Soc. Edinburgh Sect. A, Math.
126 (1996) 153--165,


\bibitem{harlaw} R. Harvey, H.B. Lawson, Jr., \textit{Calibrated geometries}, Acta Math. 148 (1982), 47--157,

\bibitem{mariant} M. Ionel, \textit{Second order families of special Lagrangian submanifolds in $\Cx^4$}, J. Differential Geometry 65 (2003), 211--272,

\bibitem{kriele} M. Kriele, L. Vrancken, \textit{An extremal class of 3-dimensional hyperbolic affine spheres}, Geometriae Dedicata 77, 239--252,  (1999),

\bibitem{hiepko} S. N\"olker, \textit{Isometric immersions of warped Products}, Differential Geometry and its Applications 6 (1996), 1–-30,

\bibitem{nomkob1} K. Nomizu-S. Kobayashi, Foundations of Differential Geometry, Vol. I, Interscience Publishers (1963),

\bibitem{nomkob2} K. Nomizu, S. Kobayashi, Foundations of Differential Geometry, Vol. II, Interscience Publishers (1969),

\bibitem{ziegel} H. Reckziegel, \textit{Horizontal lifts of isometric immersions into the bundle space of a pseudo-Riemannian submersion}, Global Differential Geometry and global Analysis, 1984, 264--279; Lecture Notes in Mathematics 1156 (1985), Springer Verlag,

\bibitem{vranck} L. Vrancken, \textit{Special classes of 3-dimensional affine hyperpheres characterised by properties of their cubic form}, Contemporary Geometry and related topics, 431--459, World Sci. Publ., River Edge, NJ (2004).


\end{thebibliography}
 \end{document}